\documentclass[final]{dmtcs-episciences}
\usepackage[utf8]{inputenc}
\usepackage{graphicx,amsmath,amsfonts,amsthm,amssymb,paralist}
\hypersetup{%
	colorlinks,
	linkcolor={black},
	citecolor={black},
	urlcolor={black},
	pdftitle={Separating Layered Treewidth and Row Treewidth},
	pdfauthor={Prosenjit Bose, Vida Dujmović, Mehrnoosh Javarsineh, Pat Morin, David R. Wood}} 
\usepackage[noabbrev,capitalise]{cleveref}
\usepackage[usenames,dvipsnames,svgnames,table]{xcolor}
\usepackage[shortlabels]{enumitem}
\setlist[itemize]{topsep=0ex,itemsep=0ex,parsep=0.4ex}
\setlist[enumerate]{topsep=0ex,itemsep=0ex,parsep=0.4ex}
\usepackage[longnamesfirst,numbers,sort&compress]{natbib}
\let\ge\geqslant
\let\leq\leqslant
\let\geq\geqslant

\let\le\leqslant

\DeclareMathOperator{\qn}{qn}
\DeclareMathOperator{\tw}{tw}
\DeclareMathOperator{\ltw}{ltw}
\DeclareMathOperator{\rtw}{rtw}
\DeclareMathOperator{\pw}{pw}
\DeclareMathOperator{\lpw}{lpw}
\DeclareMathOperator{\rpw}{rpw}
\newcommand{\GG}{\mathcal{G}}
\newcommand{\PP}{P_{\hspace{-0.2ex}\infty}}
\newcommand{\CartProd}{\mathbin{\square}}
\theoremstyle{plain}
\newtheorem{thm}{Theorem}
\crefname{thm}{Theorem}{Theorem}
\newtheorem{lem}[thm]{Lemma}
\crefname{lem}{Lemma}{Lemma}
\newtheorem{cor}[thm]{Corollary}
\crefname{cor}{Corollary}{Corollary}
\newtheorem{obs}[thm]{Observation}
\crefname{obs}{Observation}{Observation}
\newcommand{\ceil}[1]{{\lceil #1 \rceil}}

\newcommand{\N}{\mathbb{N}}
\newcommand{\defn}[1]{\textcolor{Maroon}{\emph{#1}}}
\title[Separating Layered Treewidth and Row Treewidth]{Separating Layered Treewidth and \\
	Row Treewidth\thanks{Research of Bose, Dujmovi\'c and Morin supported by NSERC. Research of Wood supported by the Australian Research Council.}}
\author{%
Prosenjit Bose\affiliationmark{1} \and 
Vida Dujmović\,\affiliationmark{2} \and
Mehrnoosh Javarsineh\affiliationmark{1}  \newline
\and Pat Morin\affiliationmark{1} \and 
David~R.~Wood\affiliationmark{3}}
\affiliation{
School of Computer Science, Carleton University, Ottawa, Canada\\
School of Computer Science and Electrical Engineering, University of Ottawa, Ottawa, Canada\\
School of Mathematics, Monash University, Melbourne, Australia}
\keywords{treewidth, layered treewidth, row treewidth}
\received{2021-05-05}
\revised{2021-11-24}
\accepted{2022-04-12}

\begin{document}
\publicationdetails{24}{2022}{1}{18}{7458}
\maketitle

\begin{abstract}
Layered treewidth and row treewidth are recently introduced graph parameters that have been  key ingredients in the solution of several well-known open problems. In particular, the \defn{layered treewidth} of a graph $G$ is the minimum integer $k$ such that $G$ has a tree-decomposition and a layering such that each bag has at most $k$ vertices in each layer. The \defn{row treewidth} of $G$ is the minimum integer $k$ such that $G$ is isomorphic to a subgraph of $H\boxtimes P$ for some graph $H$ of treewidth at most $k$ and for some path $P$. It follows from the definitions that the layered treewidth of a graph is at most its row treewidth plus 1. Moreover, a minor-closed class has bounded layered treewidth if and only if it has bounded row treewidth. However, it has been open whether row treewidth is bounded by a function of layered treewidth. This paper answers this question in the negative. In particular, for every integer $k$ we describe a graph with layered treewidth 1 and row treewidth $k$. We also prove an analogous result for layered pathwidth and row pathwidth.
\end{abstract}

\section{Introduction}

Treewidth is a graph parameter that measures how similar a given graph is to a tree; it is of fundamental importance in structural and algorithmic graph theory; see the surveys \citep{Bodlaender-TCS98,Reed03,HW17}.

Layered treewidth is a variant of treewidth introduced independently by \citet{DMW17} and \citet{Shahrokhi13}; see \cref{Definitions} for the definition. A key property is that layered treewidth is bounded on planar graphs but treewidth is not. In particular, planar graphs have layered treewidth at most 3 \citep{DMW17}, but the $n\times n$ grid graph has treewidth $n$. Layered treewidth has been used in upper bounds on several graph parameters including queue-number~\citep{DMW17,DF18}, stack number~\citep{DF18}, boxicity~\citep{SW20}, clustered chromatic number~\citep{LW1}, generalised colouring numbers~\cite{vdHW18}, asymptotic dimension~\citep{BBEGLPS}, as well for results in intersection graph theory \citep{Shahrokhi13}.

Row treewidth\footnote{The name ``row treewidth'' is an original contribution of the present paper.} is a refinement of layered treewidth, implicitly introduced by \citet{DJMMUW20}, who proved that planar graphs have row treewidth at most 8. This result and its generalisations have been the key to solving several open problems, regarding queue-number \citep{DJMMUW20}, nonrepetitive chromatic number~\citep{DEJWW20}, universal graphs~\cite{EJM,BGP20,DEJGMM21}, centred colouring~\citep{DFMS21}, graph drawing~\citep{DJMMUW20,Pupyrev20}, and vertex ranking \citep{BDJM}.

Layered and row treewidth are closely related in that the layered treewidth of a graph is at most its row treewidth plus 1, and a minor-closed class has bounded layered treewidth if and only if it has bounded row treewidth. However, a fundamental open problem is whether row treewidth is bounded by a function of layered treewidth. This paper answers this question in the negative.

\begin{thm}
\label{sep-ltw-etw}
For each $k\in\N$ there is a graph with layered treewidth 1 and row treewidth $k$.
\end{thm}

This result is proved in \cref{sep-tw}. In \cref{Queues} we use a key lemma from the proof of \cref{sep-ltw-etw} to prove a result about the relationship between layered treewidth and queue-number. 

Layered pathwidth is a graph parameter analogous to layered treewidth, first studied by \citet{bannister.devanny.ea:track} and \citet{DEJMW20}. Row pathwidth is defined in a similar way to row treewidth. We prove the following analogue of \cref{sep-ltw-etw} for layered pathwidth and row pathwidth.

\begin{thm}
\label{sep-lpw-epw}
For each $k\in\N$ there is a tree with layered pathwidth 1 and row pathwidth $k$.
\end{thm}

\section{Definitions}
\label{Definitions}

We consider finite, undirected, simple graphs and use standard graph theory terminology~\citep{diestel:graph}.

\subsubsection*{Minors}

A graph $H$ is a \defn{minor} of a graph $G$ if a graph isomorphic to $H$ can be obtained from a subgraph of $G$ by contracting edges. A graph class $\GG$ is \defn{minor-closed} if for every graph $G\in\GG$, every minor of $G$ is in $\GG$. A graph class $\GG$ is \defn{proper} if some graph is not in $\GG$. A \defn{$K_t$ model} in a graph $G$ is a set $\{X_1,\dots,X_t\}$ of pairwise-disjoint connected subgraphs in $G$, such that there is an edge of $G$ between $X_i$ and $X_j$ for all distinct $i,j\in\{1,\dots,t\}$. Clearly, $K_t$ is a minor of $G$ if and only if $G$ contains a $K_t$ model.

\subsubsection*{Treewidth and Pathwidth}

A \defn{tree decomposition} $\mathcal{T}$ of a graph $G$ is a collection $(B_x:x\in V(T))$ of subsets of $V(G)$ called \defn{bags} indexed by the nodes of a tree $T$ such that
\begin{inparaenum}[(i)]
	\item for each $v\in V(G)$, the induced subgraph $T[\{x\in V(T):v\in B_x\}]$ is nonempty and connected; and
	\item for each edge $vw\in E(G)$, there exists $x\in V(T)$ such that $\{v,w\}\subseteq B_x$.
\end{inparaenum}
The \defn{width} of a tree-decomposition is the size of its largest bag, minus $1$.  The \defn{treewidth $\tw(G)$} of a graph $G$ is the minimum width of any tree-decomposition of $G$. If $\mathcal{P}=(B_x:x\in V(P))$ is a tree-decomposition of $G$ and $P$ is a path, then $\mathcal{P}$ is a \defn{path-decomposition} of $G$.  The \defn{pathwidth $\pw(G)$} of a graph $G$ is the minimum width of any path-decomposition of $G$. For each $k\in\N$, the graphs with treewidth at most $k$ form a minor-closed class, as do the graphs with pathwidth at most $k$. Note that $\tw(K_n)=\pw(K_n)=n-1$.

\subsection*{Layered Treewidth and Pathwidth}

A \defn{layering} $\mathcal{L}$ of a graph $G$ is a partition of $V(G)$ into a sequence of sets $(L_0,L_1,\ldots)$ such that for any edge $vw\in E(G)$, if $v\in L_i$ and $w\in L_j$ then $|i-j|\le 1$. For example, if $r$ is a vertex in a connected graph $G$, and $L_i$ is the set of vertices at distance $i$ from $r$ in $G$, then $(L_0,L_1,\dots)$ is the  \defn{breadth-first} layering of $G$ rooted at $r$. A \defn{layered tree-decomposition} of a graph $G$ consists of a pair $(\mathcal{L},\mathcal{T})$ where $\mathcal{L}=(L_0,L_1,\ldots)$ is a layering of $G$ and $\mathcal{T}=(B_x:x\in V(T))$ is a tree-decomposition of $G$.  The (layered) \defn{width} of $(\mathcal{L},\mathcal{T})$ is the size of the largest intersection between a layer and a bag; that is, $\max\{|B_x\cap L_i|:x\in V(T), i\in\N_0\}$.  The \defn{layered treewidth $\ltw(G)$} of $G$ is the minimum width of any layered tree-decomposition of $G$.

This definition was introduced independently by \citet{Shahrokhi13} and \citet{DMW17}. The latter authors proved that $\ltw(G) \leq 3$ for every planar graph $G$; more generally that $\ltw(G) \leq 2g+3$ for every graph $G$ of Euler genus $g$; and that a minor-closed class $\GG$ has bounded layered treewidth if and only if some apex graph\footnote{A graph $G$ is \defn{apex} if $G-v$ is planar for some vertex $v$. A graph $G$ is an \defn{apex-forest} if $G-v$ is a forest for some vertex $v$.} is not in $\GG$. For an arbitrary proper minor-closed class $\GG$, \citet{DMW17} showed that every graph in $\GG$ has a tree-decomposition in which each bag has a bounded set of vertices whose deletion leaves a subgraph with bounded layered treewidth. This version of Robertson and Seymour's Graph Minor Structure Theorem has proved to be very useful \citep{LW1,LW2a,LW3,BBEGLPS}.

If $(\mathcal{L},\mathcal{P})$ is a layered tree-decomposition of $G$ and $\mathcal{P}$ is a path-decomposition, then $(\mathcal{L},\mathcal{P})$ is a \defn{layered path-decomposition} of $G$.  The \defn{layered pathwidth} of $G$ is the minimum width of any layered path-decomposition of $G$. This parameter was introduced by \citet{bannister.devanny.ea:track}, who proved that every outerplanar graph has layered pathwidth at most 2 (amongst other examples).

\subsection*{Row Treewidth and Pathwidth}

The \defn{cartesian product} of graphs $A$ and $B$, denoted by $A\CartProd B$, is the graph with vertex set $V(A)\times V(B)$, where distinct vertices $(v,x),(w,y)\in V(A)\times V(B)$ are adjacent if:
$v=w$ and $xy\in E(B)$; or
$x=y$ and $vw\in E(A)$.
The \defn{direct product} of $A$ and $B$, denoted by $A\times  B$, is the graph with vertex set $V(A)\times V(B)$, where distinct vertices $(v,x),(w,y)\in V(A)\times V(B)$ are adjacent if $vw\in E(A)$ and $xy\in E(B)$.
The \defn{strong product} of $A$ and $B$, denoted by $A\boxtimes B$, is the union of $A\CartProd B$ and $A\times B$.

The \defn{row treewidth $\rtw(G)$} of a graph $G$ is the minimum treewidth of a graph $H$ such that $G$ is isomorphic to a subgraph of $H\boxtimes \PP$,
where $\PP$ is the 1-way infinite path. This definition is implicit in the work of \citet{DJMMUW20} who proved that $\rtw(G) \leq 8$ for every planar graph $G$; more generally, that $\rtw(G) \leq 2g+9$ for every graph $G$ of Euler genus $g$; and that a minor-closed class $\GG$ has bounded row treewidth if and only if some apex graph is not in $\GG$. For an arbitrary minor-closed class $\GG$, \mbox{\citet{DJMMUW20}} showed that every graph in $\GG$ has a tree-decomposition in which each bag has a bounded set of vertices whose deletion leaves a subgraph with bounded row treewidth.

It follows from the definitions that for every graph $G$,
$$\ltw(G) \leq \rtw(G)+1.$$
To see this, suppose that $G\subseteq H\boxtimes P_\infty$ where $\tw(H)=\rtw(G)$. 
Let $\mathcal{T}=(B_x:x\in V(T))$ be a minimum-width tree-decomposition of $H$. 
Assume $V(P_\infty)=\N$. 
For each $b\in \N$, 
let $L_b := \{ (a,b)\in V(G) : a\in V(H)\}$. 
So $\mathcal{L}=(L_1,L_2,\dots)$ is a layering of $G$.
For each $x\in V(T)$, 
let $B'_x := \{ (a,b)\in V(G) : a\in V(H)\cap B_x,\, b\in\N \}$. 
So $\mathcal{T'}:=(B'_x:x\in V(T))$ is a tree-decomposition of $G$. 
Note that $|B'_x\cap L_b| \leq |B_x|\leq\tw(H)+1$. 
Thus $(\mathcal{L},\mathcal{T})$ is a layered tree-decomposition of $G$ with width at most $\tw(H)+1=\rtw(G)+1$.

Define the \defn{row pathwidth} $\rpw(G)$ of a graph $G$ to be the minimum pathwidth of a graph $H$ such that $G$ is isomorphic to a subgraph of $H\boxtimes \PP$. Then $\lpw(G) \leq \rpw(G)+1$ for every graph $G$.
As explained above, ``bounded layered treewidth'' and ``bounded row treewidth'' coincide for minor-closed classes. However, this is not the case for their pathwidth analogues. \citet{DEJMW20} proved that a minor-closed class $\GG$ has bounded layered pathwidth if and only if some apex-forest is not in $\GG$. However, forests are a minor-closed class excluding some apex-forest (namely, $K_3$), but forests have unbounded row pathwidth by \cref{sep-lpw-epw}. In fact, \citet{RS-I} proved that for every fixed forest $T$, the class of $T$-minor-free graphs has bounded pathwidth (see \citep{BRST91} for a tight bound). It follows that a minor-closed class $\GG$ has bounded row pathwidth if and only if $\GG$ has bounded pathwidth if and only if some tree is not in $\GG$.

\section{Separating Layered Treewidth and Row Treewidth}
\label{sep-tw}


This section proves \cref{sep-ltw-etw}. Subdivisions play a key role. A \defn{subdivision} of a graph $G$ is a graph $G'$ obtained by replacing each edge $vw$ of $G$ with a path $P_{vw}$ from $v$ to $w$ whose internal vertices have degree $2$.  If each $P_{vw}$ has exactly $s$ internal vertices, then $G'$ is the \defn{$s$-subdivision} of $G$.  If each $P_{vw}$ has at most $s$ internal vertices, then $G'$ is a \defn{$(\le\! s)$-subdivision} of $G$.
It is well known and easily proved that $\tw(G')=\tw(G)$ for every subdivision $G'$ of $G$.

The proof of \cref{sep-ltw-etw} is based on two lemmas. The first shows that subdivisions efficiently reduce layered treewidth (\cref{LayeredTreewidthSubdivision}). The second shows that subdivisions do not efficiently reduce row treewidth (\cref{LargeRowTreewidthSubdiv}). The theorem quickly follows.

\subsection{Subdivisions Efficiently Reduce Layered Treewidth}

\begin{lem}
\label{LayeredTreewidthSubdivision}
For every graph $G$ with layered treewidth $k\in\N$,
\begin{enumerate}[label=(\alph*),font=\upshape]
\item there exists a $(\le\!2k-2)$-subdivision $G'$ of $G$ of layered
treewidth 1; and
\item if any subdivision $G'$ of $G$ has layered treewidth at most $c$, then some edge of $G$ is subdivided at least $k/c-1$ times in $G'$. 
\end{enumerate}
\end{lem}

\begin{proof}
Let $(B_x:x\in V(T))$ be a tree-decomposition of $G$ and let
$(V_0,V_1,\dots)$ be a layering of $G$, such that $|B_x\cap V_i|\leq
k$ for each $x\in V(T)$ and $i\in\N_0$.

First we prove (a). 
We may assume that $B_x\cap V_i$ is a clique for each $x\in V(T)$ and $i\in\N_0$. 
So $G[V_i]$ is a chordal graph with no $(k+1)$-clique, which is therefore $k$-colourable. 
Let $c:V(G)\to\{0,1,\dots,k-1\}$ be a function such that $c$ is a proper $k$-colouring of $G[V_i]$ for each $i\in\N_0$. Thus for each $x\in V(T)$ and $i\in\N_0$, and for all distinct vertices $v,w\in B_x\cap V_i$, we have $c(v)\neq c(w)$.
Let $L_{ki+j} := \{v \in V_i : c(v)=j\}$ for $i\in\N_0$ and $j\in\{0,1,\dots,k-1\}$. Let
$G'$ be obtained from $G$ as follows. Consider each edge $e = vw$ of
$G$. Say $v\in V_i$ and $w\in V_{i'}$, and $v\in L_a$ and $w\in
L_{a'}$. Without loss of generality, $a<a'$. Then $a' \leq k(i'+1) - 1
\leq k(i+2)-1 \leq a + 2k-1$.
Replace $e$ by the path $(v, s_{e,a+1}, s_{e,a+2}, \dots, s_{e,a'-1},
w)$ in $G'$.
The number of division vertices is $a'-1-a \leq 2k-2$. 
Put the division vertex $s_{e,b}$ in $L_b$ for each $b\in\{a+1,\dots,a'-1\}$. 
So $(L_1,L_2,\dots)$ is a layering of $G'$.
Some bag $B_x$ contains $v$ and $w$.
Add a leaf node to $T$ adjacent to $x$ with corresponding bag
$\{v, s_{e,a+1}, s_{e,a+2}, \dots, s_{e,a'-1}, w\}$.
We obtain a tree-decomposition of $G'$ with at most one vertex in each
layer and in each bag. Hence $\ltw(G')=1$.


We now prove (b). Suppose that some $(\leq\!r)$-subdivision $G'$ of $G$
has $\ltw(G') \leq c$. Let $(B_x:x\in V(T))$ be a tree-decomposition
of $G'$ and let $(V_0,V_1,\dots)$ be a layering of $G'$ such that
$|B_x\cap V_i| \leq c$ for each $x\in V(T)$ and $i\in\N_0$. Orient each
edge of $G$ arbitrarily. For each oriented edge $vw$ of $G$ and for
each division vertex $z$ of $vw$, let $\alpha(z):= v$. For each node
$x\in V(T)$, let $C_x$ be obtained from $B_x$ by replacing each
division vertex $z\in B_x$ by $\alpha(z)$. Observe that $(C_x:x\in
V(T))$ is a tree-decomposition of $G$. For $j\in\N_0$, let $L_j:= V(G)
\cap ( V_{j(r+1)}\cup V_{j(r+1)+1} \cup \dots \cup V_{(j+1)(r+1)-1}
)$. Consider an edge $vw$ of $G$ with $v\in V_i$ and $w\in V_{i'}$ and
$i\leq i'$. Then $i' \leq i+ r+1$ since $vw$ is subdivided at most $r$
times. Say $v\in L_j$ and $w\in L_{j'}$. By definition, $j(r+1) \leq i
\leq (j+1)(r+1)-1$ and $j'(r+1) \leq i' \leq (j'+1)(r+1)-1$. Hence
$j'(r+1) \leq i' \leq i+r+1 \leq (j+1)(r+1)-1 + (r+1) = (j+1)(r+1) +
r$, implying $j'\leq j+1$. That is, $(L_0,L_1,\dots)$ is a layering of
$G$. Each layer $L_j$ contains at most $c(r+1)$ vertices in each bag
$C_x$. Thus $k=\ltw(G) \leq c(r+1)$, implying $r \geq k/c -1$.
\end{proof}


Note that the proof of \cref{LayeredTreewidthSubdivision}(a) is easily
adapted to show that if $s_e\geq 2k-2$ for each edge $e\in E(G)$ and
$G'$ is the subdivision of $G$ in which each edge $e$ is subdivided
$s_e$ times, then $\ltw(G')\leq 2$. We omit these straightforward
details.

\citet{bannister.devanny.ea:track} proved that $\lpw(G) \leq \ceil{( \pw(G)+1)/2}$. An analogous proof shows that
$\ltw(G) \leq \ceil{( \tw(G)+1)/2}$.
\cref{LayeredTreewidthSubdivision}(a) then implies:

\begin{cor}\label{treewidth}\label{TreewidthSubdivision}
Every graph with treewidth $k$ has a $(\leq k)$-subdivision with layered treewidth 1.
\end{cor}

We remark that \cref{treewidth} is tight, up to a small constant factor:
\begin{obs}
    Let $G'$ be any $(\le\!s)$-subdivision of the complete graph $K_{k+1}$.  Then $\ltw(G')\ge (k+1)/(2s+3)$. In particular, $\ltw(G')>1$ if $s < (k-2)/2$.
\end{obs}

\begin{proof}
    Let $(\mathcal{L},\mathcal{T})$ be a layered tree-decomposition of $G'$.    Since $K_{k+1}$ has radius 1, $G'$ has radius at most $s+1$ and therefore $|\mathcal{L}|\le 2s+3$.  Since $K_{k+1}$ has treewidth $k$, so does $G'$. Thus $\mathcal{T}$ has at least one bag $B_x$ of size at least $k+1$.  By the pigeonhole principle, $|B_x\cap L|\ge (k+1)/\mathcal{|L|}\ge (k+1)/(2s+3)$ for at least one $L\in\mathcal{L}$.
\end{proof}

\subsection{Subdivisions do not Efficiently Reduce Row Treewidth}

This section proves the following result.

\begin{lem}\label{LargeRowTreewidthSubdiv}
For each $k\in\N$ and $s\in\N_0$ there exists a graph $G$ such that for every $(\le\!s)$-subdivision $G'$ of $G$, $$\rtw(G') \geq k \geq \tw(G).$$
\end{lem}

A weaker version (with a more indirect proof) of \cref{LargeRowTreewidthSubdiv} is implied by existing results about the \defn{$p$-centred chromatic number $\chi_p(G)$} of a graph $G$; see \citep{DFMS21,joret.pilipczuk.ea:two} for the definition (we will not need it). \mbox{\citet{joret.pilipczuk.ea:two}} show that, for each $p\in\N$, there exists an integer $k\in \Theta(\sqrt{p})$ and a treewidth-$k$ graph $G$, such that if $G'$ is the $6k$-subdivision of $G$, then $$\chi_p(G') \geq 2^{\Omega(\tw(G))}.$$
On the other hand, \citet{DFMS21} (see also \citep{dujmovic.morin.ea:structure}) show that for every graph $G$ and $p\in\N$, $$\chi_p(G) \leq (p+1)p^{\binom{p+\rtw(G)}{\rtw(G)}}.$$
Thus, for the graph $G'$ constructed by \citet{joret.pilipczuk.ea:two},
$$ 2^{\Omega(\tw(G))} \leq \chi_p(G') \leq (p+1)p^{\binom{p+\rtw(G')}{\rtw(G')}} .$$
This implies that $\rtw(G') \in \Omega(\tw(G)/\log(\tw(G)))$, as noted in \citep{dvorak.huynh.ea:notes}.

To prove \cref{LargeRowTreewidthSubdiv} it will be convenient to use the language of $H$-partitions from \citep{DJMMUW20}. An \defn{$H$-partition} of a graph $G$ is a partition $\mathcal{H}=(B_x:x\in V(H))$ of $V(G)$ indexed by the nodes of some graph $H$ such that for any edge $vw\in E(G)$, if $v\in B_x$ and $w\in B_y$, then $xy\in E(H)$ or $x=y$.
A \defn{layered $H$-partition} $(\mathcal{H},\mathcal{L})$ of a graph $G$ consists of an $H$-partition $\mathcal{H}$ and a layering $\mathcal{L}$ of $G$.  The (layered) \defn{width} of $(\mathcal{H},\mathcal{L})$ is $\max\{|B_x\cap L|:x\in V(H), L\in\mathcal{L}\}$. The \defn{layered width} of an $H$-partition $\mathcal{H}$ of a graph $G$ is the minimum width, taken over all layerings $\mathcal{L}$ of $G$, of the width of $(\mathcal{H},\mathcal{L})$. \citet{DJMMUW20} observed that:

\begin{obs}[\citep{DJMMUW20}]
\label{ProductPartition}
For all graphs $G$ and $H$, $G$ is isomorphic to a subgraph of $H\boxtimes \PP \boxtimes K_w$ if and only if $G$ has an $H$-partition of layered width at most $w$. In particular, $\rtw(G)$ equals the minimum treewidth of a graph $H$ for which $G$ has a layered $H$-partition of layered width $1$.
\end{obs}



This observation and the next lemma (with $w=1$) implies \cref{LargeRowTreewidthSubdiv}. \cref{LayeredPartitionSubdivision} generalises a result of \citet{DJMMUW20} who proved the $s=0$ case.

\begin{lem}\label{LayeredPartitionSubdivision}
For all $w,k\in\N$ and $s\in\N_0$, there exists a graph $G=G_{s,k,w}$ such that $\tw(G)\leq k$, and for any $(\le\!s)$-subdivision $G'$ of $G$, and for any graph $H$ and any $H$-partition of $G'$ with layered width at most $w$, there is a $K_{k+1}$ minor in $H$, implying $\tw(H)\geq k$ and $\rtw(G) \geq k$. 
\end{lem}

Say a $K_k$ model $\{Y_1,\dots,Y_k\}$ in a graph $G$ \defn{respects} an $H$-partition $(A_x:x\in V(H))$ of $G$ if for each $x\in V(H)$ there is at most one value of $i\in\{1,\dots,k\}$ for which $V(Y_i)\cap A_x\neq \emptyset$. For each $i\in\{1,\dots,k\}$, let $X_i$ be the subgraph of $H$ induced by those vertices $x\in V(H)$ such that $V(Y_i)\cap A_x\neq\emptyset$. Then $\{X_1,\dots,X_k\}$ is a $K_k$ model in $H$. Thus the next lemma implies \cref{LayeredPartitionSubdivision}.


\begin{lem}\label{IndHyp}
For all $w\in\N$ and $k,s\in\N_0$, there exists a graph $G=G_{s,k,w}$ such that $\tw(G)\leq k$, and for any $(\le\!s)$-subdivision $G'$ of $G$, and for any graph $H$ and any $H$-partition $(A_x:x\in V(H))$ of $G'$ with layered width at most $w$, there is a model $\{Y_1,\dots,Y_{k+1}\}$ of $K_{k+1}$ in $G'$ that respects $(A_x:x\in V(H))$, and for each $i\in\{1,\dots,k+1\}$ we have $|V(Y_i)|\leq k(2s+1)+1$ and $V(Y_i) \cap V(G)\neq\emptyset$.
\end{lem}

\begin{proof}
The proof is by induction on $k$. For the base case with $k=0$, let $G_{s,0,w}$ be the graph with one vertex $v$ and no edges. Any subdivision $G'$ contains a model $Y_1=G'[\{v\}]$ of $K_1$ with $|Y_1| \leq 1 = 0(2s+1)+1$, and this model trivially respects any $H$-partition of $G'$, for any graph $H$.



Now assume $k\ge 1$, and that the induction hypothesis holds for $k-1$. Let $Q=G_{s,k-1,w}$ be the graph obtained by induction. Let $N:= k^2(2s+1)(4s+5)w+1$. Create the graph $G$ by starting with $N$ disjoint copies $Q_1,\ldots,Q_N$ of $Q$.  Next add a vertex $v$ and, for each $i\in\{1,\ldots,N\}$ and each $u\in V(Q_i)$, add $N$ internally disjoint paths of length $2$ from $v$ to $u$.

First we show that $\tw(G)\le k$. If $k=1$ then $G$ is a subdivided star, which has treewidth 1. Now assume that $k\geq 2$. Begin with any tree-decomposition $(B_x:x\in V(T))$ of $Q_1\cup\ldots\cup Q_N$
with width at most $(k-1)$. Add $v$ to every bag $B_x$.  For all $i,j\in\{1,\ldots,N\}$ and for each vertex $u\in V(Q_i)$, choose a bag $B_x$ that contains $u$, and attach a leaf node adjacent to $x$ whose bag contains $v$, $u$, and the degree-2 vertex of the $j$-th length-2 path from $v$ to $u$.  Each bag of this decomposition has size at most $\max\{3,k+1\}\le k+1$ and therefore $\tw(G)\le k$.

Let $G'$ be a $(\le\!s)$-subdivision of $G$. Let $H$ be a graph and let $(\mathcal{H},\mathcal{L})$ be a layered $H$-partition of $G'$ with width at most $w$, where $\mathcal{H}:=(A_x:x\in V(H))$.  Since the radius of $G$ is $2$, the radius of $G'$ is at most $2s+2$.  Therefore $|\mathcal{L}|\le 4s+5$, and $|A_x|\leq (4s+5)w$ for each $x\in V(H)$.

Let $z$ be the unique node of $H$ such that $v\in A_z$, and let $Q'_1,\ldots,Q'_N$ be the (possibly subdivided) copies of $Q_1,\ldots,Q_N$ that appear in $G'$. So $A_z \cap V(Q'_i)\neq\emptyset$ for at most $(4s+5)w-1$ values of $i\in\{1,\dots,N\}$.  Since $N \ge (4s+5)w$, we have $V(Q'_i)\cap A_z=\emptyset$ for some $i\in\{1,\ldots,N\}$.

Let $H_i$ be the subgraph of $H$ induced by the nodes $\tau\in V(H)$ such that $A_{\tau}\cap V(Q'_i)\neq\emptyset$. So $(A_\tau \cap V(Q'_i): \tau\in V(H_i))$ defines a width-$w$ layered $H_i$-partition of $Q'_i$ (with respect to layering $\mathcal{L}$). By induction, there is a $K_{k}$ model $\{Y_1,\dots,Y_{k}\}$ in $G'$ that respects $(A_\tau\cap V(Q'):\tau \in V(H_i))$, and for each $j\in\{1,\dots,k\}$ we have $|V(Y_j)|\leq (k-1)(2s+1)+1 \leq k(2s+1)$ and $V(Y_j) \cap V(G)\neq\emptyset$. Let $y_j$ be a vertex in $V(Y_j)\cap V(G)$. Note that $|V(Y_1\cup\dots\cup Y_k)| \leq k^2(2s+1)$. Let $F:=\bigcup\{ A_x: A_x \cap V(Y_1\cup\dots\cup Y_k) \neq\emptyset\}$. So $|F| \leq k^2(2s+1)(4s+5)w$.

Since  $V(Q'_i)\cap A_z=\emptyset$, we have $v\not\in F$. Since $N>|F|$, for each $j\in\{1,\dots,k\}$, for at least one of the $N$ paths between $v$ and $y_j$ added in the construction of $G$, the corresponding path in $G'$ avoids $F$. Let $P_j$ be this path, not including $y_j$. Let $Y_{k+1}:=\bigcup\{ P_j: j \in \{1,\dots,k\}\}$. So $V(Y_{k+1}) \cap F=\emptyset$. By construction, there is an edge from $y_j$ to $Y_{k+1}$ for each $j\in\{1,\dots,k\}$. So $\{Y_1,\dots,Y_{k+1}\}$ is a $K_{k+1}$ model in $G'$. Since $\{Y_1,\dots,Y_{k}\}$ respects $(A_\tau\cap V(Q'):\tau \in V(H_i))$ and $V(Y_{k+1}) \cap F=\emptyset$, it follows that $\{Y_1,\dots,Y_{k+1}\}$ respects $(A_x:x\in V(H))$. By construction, $|V(Y_{k+1})| \leq k(2s+1)+1$ and $v\in V(Y_{k+1})\cap V(G)$. By assumption,
for each $j\in\{1,\dots,k\}$ we have $|V(Y_j)|\leq (k-1)(2s+1)+1$ and $V(Y_j) \cap V(G)\neq\emptyset$. This shows that $\{Y_1,\dots,Y_{k+1}\}$ is the desired $K_{k+1}$ model in $G'$.
\end{proof}

We now prove the following strengthening of \cref{sep-ltw-etw}. 
	
\begin{thm}
For each $k\in\N$ there is a graph $G$ with $\ltw(G)=1$ and $\rtw(G)=\tw(G)=k$.
\end{thm}

\begin{proof}
By \cref{LargeRowTreewidthSubdiv} with $s=k$, there is a graph $G$ with $\tw(G)\leq k$ such that $\rtw(G')\geq k$ for every $(\le\!k)$-subdivision $G'$ of $G$. 
By \cref{TreewidthSubdivision}, there exists a $(\le\!k)$-subdivision $G'$ of $G$ with $\ltw(G')=1$. By definition, $\rtw(H) \leq \tw(H)$ for every graph $H$. It is well known and easily proved that $\tw(H)=\tw(H')$ for every subdivision $H'$ of $H$. Thus $k \leq \rtw(G')
 \leq \tw(G') = \tw(G) \leq k$, implying $\rtw(G)=\tw(G')=k$. 
\end{proof}

See \cite{GKRSS18,joret.pilipczuk.ea:two} for other examples where $O(\tw(G))$-subdivisions of graphs $G$ are used to prove lower bounds.

\section{Queue Layouts and Layered Treewidth}
\label{Queues}

The \defn{queue-number $\qn(G)$} of a graph $G$ is a well-studied graph parameter introduced by Heath, Leighton and Rosenberg~\cite{HLR92,HR92}; we omit the definition since we will not need it. Queue-number can be upper bounded in terms of layered treewidth and row treewidth. In particular, \citet{DMW17} proved that $n$-vertex graphs of bounded layered treewidth have queue-number in $O(\log n)$, while \citet{DJMMUW20} proved that graphs of bounded row treewidth have bounded queue-number. (Row treewidth was discovered as a tool for proving that planar graphs have bounded queue-number.)\ This is a prime example of a difference in behaviour between layered treewidth and row treewidth. Nevertheless, it is open whether graphs of bounded layered treewidth have bounded queue-number. We show that the answer to this question depends entirely on the case of layered treewidth 1. 

\begin{cor}
Graphs of bounded layered treewidth have bounded queue-number if and only if graphs of layered treewidth 1 have bounded queue-number. 
\end{cor}

\begin{proof}
The forward implication is immediate. Now assume that every graph of layered treewidth 1 has queue-number at most some constant $c$. 
Let $G$ be a graph with layered treewidth $k$. By \cref{LayeredTreewidthSubdivision}, $G$ has a $(\leq k)$-subdivision $G'$ with layered treewidth 1. So $\qn(G') \leq c$. \citet{DujWoo05} proved that for every graph $H$ and $(\leq s)$-subdivision $H'$ of $H$, we have $\qn(H)\in O( \qn(H')^{2s} )$. This bound was improved to $O( \qn(H')^{s+1} )$ in \citep{DMW19}. Thus $\qn(G)\leq O(c^{k+1})$, implying that graphs of bounded layered treewidth have bounded queue-number.
\end{proof}

This proof highlights the value of considering the behaviour of a graph parameter on subdivisions. An analogous result holds for nonrepetitive chromatic number $\pi(G)$ (using a result of \citet{NOW11} to bound $\pi(H)$ in terms of $\pi(H')$). 
\begin{cor}
Graphs of bounded layered treewidth have bounded nonrepetitive chromatic number if and only if graphs of layered treewidth 1 have bounded nonrepetitive chromatic number. 
\end{cor}


\section{Separating Layered Pathwidth and Row Pathwidth}
\label{sep-pw}

Recall that \cref{sep-lpw-epw} asserts that for all $k\in\N$ there is a tree $T$ with $\lpw(T)=1$ and $\rpw(T) \geq k$. We first show that this theorem, in fact, follows from results in the literature. \citet{bannister.devanny.ea:track} noted that $\lpw(T)=1$ for every tree $T$. 
\citet{dvorak.huynh.ea:notes} showed that any family $\mathcal{G}$ of graphs with bounded row pathwidth has polynomial growth. The family of complete binary trees does not have polynomial growth, so for each $k\in\N$, there exists a complete binary tree $T$ with $\rpw(T)\ge k$. This proves \cref{sep-lpw-epw}.

We now prove the following stronger result.

\begin{thm}
\label{tree-lower-bound}
For every $k\in\N$ there exists a tree $T$ with $\pw(T)=\rpw(T)=k$.
\end{thm}

Let $T_{d,h}$ be the complete $d$-ary tree of height $h$.
It is folklore that $\pw(T_{d,h})=h$ for all $d\geq 3$ (see \cite{Scheffler90,EST-IC94,Sch92}). So \cref{ProductPartition,tree-lemma} with $\ell=3$ and $w=1$ implies \cref{tree-lower-bound}.

\begin{lem}
\label{tree-lemma}
For all $h\in\N_0$ and $w,\ell\in\N$ there exists $d\in\N$ such that for every graph $H$ and every $H$-partition $(A_x:x\in V(H))$ of $T_{d,h}$ with layered width at most $w$, the graph $H$ contains a subgraph isomorphic to $T_{\ell,h}$. Moreover, if $r$ is the root vertex of $T_{d,h}$ and $r\in A_z$, then $z$ is the root vertex of a subgraph of $H$ isomorphic to $T_{\ell,h}$.
\end{lem}

\begin{proof}
We proceed by induction on $h$. The case $h=0$ is trivial. Now assume that $h\geq 1$.
We may assume that the number of layers is at most the diameter of $T_{d,h}$ plus 1, which equals $2h+1$. So $|A_x|\leq w(2h+1)$ for each $x\in V(H)$. Let $r_1,\dots,r_d$ be the children of $r$ in $T_{d,h}$. Let $T^1,\dots,T^d$ be the copies of $T_{d,h-1}$ in $T_{d,h}$, where $T^i$ is rooted at $r_i$. At most $w(2h+1)-1$ of $T^1,\dots,T^d$ intersect $A_z$. Without loss of generality,  $T^1,\dots,T^{d-w(2h+1)+1}$ do not intersect $A_z$.
For each $i\in\{1,\dots,d-w(2h+1)+1\}$, let $z_i$ be the vertex of $H$ such that $r_i\in A_{z_i}$. Since $T^i\cap A_z=\emptyset$, we have $z_i\neq z$. Since $rr_i\in E(T_{d,h})$, we have $zz_i\in E(H)$. By induction, $H-z$ contains a subgraph $S^i$ isomorphic to $T_{\ell,h-1}$ rooted at $z_i$. Let $X$ be the intersection graph of $S^1,\dots,S^{d-w(2h+1)+1}$. If $S^i$ and $S^j$ intersect in node $y$ of $H$, then $T^i$ and $T^j$ both intersect $A_y$. Since $|V(S^i)| \leq (\ell+1)^{h-1}$ and $|A_y|\leq w(2h+1)$, $X$ has maximum degree $\Delta(X)\leq w(2h+1)(\ell+1)^{h}$. Thus $\chi(X)\leq w(2h+1)(\ell+1)^{h}+1$. For sufficiently large $d$, we have $|V(X)|>(\ell-1)\chi(X)$. Thus, in any $\chi(X)$-colouring of $X$, some colour class has at least $\ell$ vertices. Without loss of generality, $S^1,\dots,S^\ell$ are pairwise-disjoint. Hence, $S_1\cup\dots\cup S_\ell$ along with $z$ forms a subgraph of $H$ isomorphic to $T_{\ell,h}$, rooted at $z$, as desired.
\end{proof}

We finish with an open problem: what is $\rpw(T_{2,h})$? It follows from a result of
\citet{dvorak.huynh.ea:notes} that $\rpw(T_{2,h})\in\Omega(\frac{h}{\log h})$. Is this tight, or is $\rpw(T_{2,h})\in\Omega(h)$? Obviously $\rpw(T_{2,h})\leq\pw(T_{2,h})=\ceil{\frac{h}{2}}$.

\acknowledgements
Thanks to Robert Hickingbotham and to both referees for several helpful comments. 

\let\oldthebibliography=\thebibliography
\let\endoldthebibliography=\endthebibliography
\renewenvironment{thebibliography}[1]{%
	\begin{oldthebibliography}{#1}%
		\setlength{\parskip}{0ex}%
		\setlength{\itemsep}{0ex}%
	}{\end{oldthebibliography}}


\end{document}